\documentclass[11pt,reqno]{amsart}
\usepackage{amsthm,amsmath,amssymb}
\usepackage[utf8]{inputenc}
\usepackage{graphicx}
\usepackage{float}
\usepackage[colorlinks=true,linkcolor=blue,urlcolor=blue,citecolor=blue]{hyperref}
\usepackage{tikz}
\usepackage{enumerate}
\usepackage[toc,page]{appendix}
\usepackage{lscape}
\usepackage{multirow}
\usepackage{adjustbox,expl3,etoolbox} 
\usepackage{longtable}
\usepackage{xcolor,afterpage}
\usepackage[margin = 3.1cm]{geometry}
\usepackage{stmaryrd}
\usepackage{longtable}
\usepackage{makecell}
\setcellgapes{4pt}

\newtheorem{thm}{Theorem}[section]
\newtheorem{lem}[thm]{Lemma}
\theoremstyle{definition}

\newtheorem{prop}[thm]{Proposition}

\newtheorem{cor}[thm]{Corollary}
\newtheorem{exa}[thm]{Example}
\newtheorem{rmk}[thm]{Remark}

\DeclareMathOperator{\Gal}{Gal}
\DeclareMathOperator{\sym}{Sym}

\title[]{The search for small association schemes with noncyclotomic eigenvalues}
{\author[Herman]{Allen Herman\textsuperscript{1,2}}
	\thanks{\textsuperscript{1} This author's work was supported by an NSERC Discovery Grant. }
	\thanks{\textsuperscript{2} Department of Mathematics and Statistics, University of Regina, Regina, Saskatchewan S4S 0A2, Canada}
	\address{Department of Mathematics and Statistics, University of Regina, Regina, Saskatchewan S4S 0A2, Canada}}\email{allen.herman@uregina.ca}
	{\author[Maleki]{Roghayeh Maleki\textsuperscript{2}}
	\email{rmaleki@uregina.ca}}
\begin{document}

\begin{abstract}
In this article we determine feasible parameter sets for (what could potentially be) commutative association schemes with noncyclotomic eigenvalues that are of smallest possible rank and order. A feasible parameter set for a commutative association scheme corresponds to a standard integral table algebra with integral multiplicities that satisfies all of the parameter restrictions known to hold for association schemes.  For each rank and involution type, we generate an algebraic set for which any suitable integral solution corresponds to a standard integral table algebra with integral multiplicities, and then try to find the smallest suitable solution. The main results of this paper show the eigenvalues of association schemes of rank $4$ and nonsymmetric  association schemes of rank $ 5$ will always be cyclotomic.  In the rank $5$ cases, the results rely on calculations done by computer for Gröbner bases or for bases of rational vector spaces spanned by polynomials.  We give several examples of feasible parameter sets for small symmetric association schemes of rank $5$ that have noncyclotomic eigenvalues. 
\end{abstract}



\date{\today}
	
	\maketitle
	
\section{Introduction}	
This paper investigates the  Cyclotomic Eigenvalue Question for commutative association schemes that was posed by Simon Norton at Oberwolfach in 1980 \cite{bannaialgebraic}.  This question asks if the eigenvalues of all the adjacency matrices of relations in the scheme lie in a cyclotomic number field, or equivalently if every entry of the character table (i.e., first eigenmatrix) of a commutative association scheme is cyclotomic.  Showing this is a straightforward exercise for association schemes of rank 2 and 3.  For commutative Schurian association schemes, this property is a consequence of the character theory of Hecke algebras and the fact that Morita equivalent algebras have isomorphic centers (see \cite{herman2011schur}).   For commutative association schemes that are both $P$- and $Q$-polynomial, it follows from the fact that the splitting field of the scheme is quadratic extension of the rationals, a key ingredient of Bang, Dubickas, Koolen, and Moulten's proof of the Bannai-Ito conjecture (\cite{bang-dubickas-koolen-moulten2015}, see also \cite{martin2009commutative}).  Herman and Rahnamai Barghi proved it for commutative quasi-thin schemes \cite{herman2008quasithin}, which were later shown by Muzychuk and Ponomarenko to always be Schurian \cite{muzychuk-ponomarenko2012}.  Herman and Rahnamai Barghi also showed the cyclotomic eigenvalue property holds for commutative association schemes whose elements have valency $\le 2$ except for possibly one element of valency $3$ and/or one element of valency $> 4$ \cite[Theorem 3.3]{herman2008quasithin}.

For association schemes in general we do not know if the character values have to be cyclotomic, but we do have noncommutative examples for which the eigenvalues are not cyclotomic -- the smallest examples are two  noncommutative Schurian association schemes of order 26, and three noncommutative Schur rings of order 32 (in the latter case the corresponding graphs are Cayley graphs on a nonabelian group of order 32).

In this article we investigate the cyclotomic eigenvalue question from a smallest counterexample perspective.  For a given rank and involution type, our approach will be to generate an algebraic set in a multivariate polynomial ring in variables corresponding to the intersection numbers and character table parameters of such an association scheme. Each suitable integer point in this algebraic set corresponds to a standard integral table algebra with integral multiplicities (SITAwIM) that has the corresponding intersection matrices and character table via its regular representation.  We use the algebraic set to search for small SITAwIMs of the given type that have some noncyclotomic eigenvalues. 

The algebraic sets themselves are not easy to work with, as they are not monomial and the number of variables and polynomial generators is too large for available computer algebra systems to do efficient Gröbner basis calculations.   After manually reducing the algebraic sets with all available linear substitutions, we search for solutions by specifying values for sufficiently many remaining parameters that the resulting algebraic set can be resolved with a Gröbner basis calculation.   Using this approach, we are able to show the answer to the cyclotomic eigenvalue question is yes for all association schemes of rank 4 and for both involution types of nonsymmetric association schemes of rank 5.   
 For commutative association schemes, the noncyclotomic eigenvalue property implies the Galois group of the splitting field will be non-Abelian, and so there must be an orbit of size at least $3$ in its action on irreducible characters.  We will say the Galois group acts {\it $k$-point transitively} if the size of its largest orbit on irreducible characters is $k$.  So for symmetric association schemes of rank 5, noncylotomic eigenvalues can only occur when the Galois group of the splitting field is $3$- or $4$-point transitive. 
When this action is $4$-point transitive, the association scheme will be pseudocyclic.  This greatly reduces the number of cases we need to consider, and our searches have been able to produce six feasible examples of orders less than $1000$, the smallest having order $249$.  When the action is $3$-point transitive, the scheme is not pseudocyclic, so the search space is much larger.  We have been able to generate all examples of order less than 100 and a few more with order less than 250, ten of which satisfy all available feasibility criteria.  The smallest of these feasible examples have order $35$, $45$, $76$, and $93$.  From the partial classification of association schemes of order $35$, we know the order $35$ example cannot be realized.  The status of the larger feasible examples is open. 

\section{Preliminaries}	
In this section, we review some background results that are needed in this work. Recall that an involution $\phi$ of a finite-dimensional algebra $A$ is a map $\phi: A\rightarrow A$ such that $\phi \circ \phi=id_{A}$.
\subsection{SITA parameters}
An {\it integral table algebra} $(A,\mathbf{B})$ is a finite-dimensional complex algebra $A$ with distinguished basis $\mathbf{B}=\{b_i \mid  i\in I=\{0,1,\ldots,r-1\}\}$ such that
\begin{enumerate}[(i)]
	\item  $1 \in \mathbf{B}$,
	\item $A$ has an involution  $*: A \rightarrow A$ that is additive, reverses multiplication, and acts as complex conjugation on scalars,
	\item $\mathbf{B}$ is $*-$invariant,
	\item $\mathbf{B}$ produces non-negative integer structure constants (see \ref{structure constants}),
	\item $\mathbf{B}$ satisfies the pseudo-inverse condition: for all $b_i,b_j \in \mathbf{B}$, the coefficient of $1$ in $b_ib_j^*$ is positive if and only if $b_j=b_i^*$.
\end{enumerate}

Note that since  $\mathbf{B}^*=\mathbf{B}$, the involution $*$ is a permutation of $\{ 0,1,\ldots,r-1 \}$. Therefore, the action of the involution $*$ can be defined by $(b_i)^*=b_{i^*}$ for all $b_i\in \mathbf{B}$.

 In order to consider $A$ as an algebra of square matrices over $\mathbb{C}$, we identify the elements of $\mathbf{B}$ with their left regular matrices in the basis $\mathbf{B}$.  The basis $\mathbf{B}$ is called {\it standard} when, for all $b_i \in \mathbf{B}$, the coefficient of $1$ in $b_ib_i^*$ is equal to the maximal eigenvalue of the regular matrix $b_i$.  We refer to $r=|\mathbf{B}|$ as the {\it rank} of the table algebra, when the basis $\mathbf{B}$ is standard we say that $(A,\mathbf{B})$ is a standard integral table algebra, or SITA.  The action of the involution $*$ on the basis $\mathbf{B}$ determines the {\it involution type} of the table algebra of a given rank. 

The adjacency algebra of an association scheme is the prototypical example of a SITA, as the defining basis of adjacency matrices is a standard basis.  Conversely, the structure constants determined by the basis of adjacency matrices of an association scheme determine a standard integral table algebra that is {\it realizable} as an association scheme.  Many open problems concerning missing combinatorial objects correspond to standard integral table algebras that satisfy all the known conditions on their parameters for being realized by an association scheme, but are yet to be actually constructed.  We call such standard integral table algebras (or their parameter sets) {\it feasible}. 

Let $P=(\chi_i(b_j))_{i,j}$ be the character table of $A$ with respect to the distinguished basis $\mathbf{B}$, whose rows are indexed by the irreducible characters of $A$ and columns are indexed by the basis $\mathbf{B}$.  As we can restrict ourselves to the commutative table algebras in this paper, $P$ will be an $r \times r$ matrix.   We order the irreducible characters so that the entries $P_{0,j}=\chi_0(b_j)=\delta_j$, $j=0,1,\dots,r-1$ are equal to the Perron-Frobenius eigenvalues of the basis matrices (i.e., the {\it degrees} of standard basis elements, or in the association scheme case, the {\it valencies} of the scheme relations).  The {\it order} of a standard integral table algebra is the sum of its degrees; that is, $n = \sum_{j=0}^{r-1} \delta_j$.  The {\it multiplicity} $m_i$ of each irreducible character $\chi_i$ can be computed by the following formula \cite{bannai1993character}
\begin{align*}
\sum_{j=0}^{r-1}\frac{|P_{ij}|^2}{\delta_j}=\frac{n}{m_i}, \mbox{ for } i=0,1,\dots,r-1.
\end{align*}
For table algebras, the multiplicity $m_i$ corresponds to the coefficient of $\chi_i$ when the standard feasible trace map $\rho(\sum_{j=0}^{r-1} \alpha_j b_j) = n\alpha_0$ is expressed as a (positive) linear combination of the irreducible characters of $A$.  We always have $m_0=1$, but the other multiplicities $m_i$ for $i=1,\dots,r-1$ are only required to be positive real numbers.  When the SITA is realized by an association scheme, the standard feasible trace is the character corresponding to the standard representation of the SITA, so the $m_i$'s will be positive integers.  This is just one of the feasibility conditions for the parameters of an association scheme.  In this way, each feasible parameter set for association schemes determines a standard integral table algebra with integral multiplicities, i.e., a SITAwIM.  

A SITA is called {\it pseudocyclic} if its multiplicities $m_i$ for $i>0$ are all equal to the same positive constant $m$.  By a result of Blau and Xu \cite{xu2011pseudocyclic}, pseudocyclic SITAs are also {\it homogeneous}, that is, all degrees $\delta_i$ for $i>0$ are equal to the same positive constant.
	
\subsection{General conditions on SITA parameters}	\label{structure constants}
Let $\mathbf{B} = \{b_0,b_1,...,b_{r-1}\}$ be the standard basis of a SITA $(A,\mathbf{B})$.  Denote the \emph{structure constants} relative to the basis $\mathbf{B}$ by $(\lambda_{ijk})_{i,j,k=0}^{r-1}$, so 
$$ b_i b_j = \sum_{k=0}^{r-1} \lambda_{ijk} b_k, \mbox{ for all } i,j \in \{0,1,\dots,r-1\}. $$
Let $\chi_0(b_i) = \delta_i$ be the degree (or valency) of the basis element $b_i \in \mathbf{B}$, for all $i \in \{0,1,\dots,r-1\}$. When the algebra has a standard basis, we have $\delta_i=\delta_{i^*}=\lambda_{i^*i0}$ (see \cite[Definition~1.3]{blau2009table}).

Associativity of $A$ and the pseudo-inverse condition on the standard basis can be used to prove two general properties of the structure constants relative to $\mathbf{B}$. 

\begin{lem}\label{conditions on Sc's}
For all $i,j,k \in \{0,1,\dots,r-1\}$, 

{\rm (i)}  $\lambda_{jki^*}\delta_i=\lambda_{kij^*}\delta_j=\lambda_{ijk^*}\delta_k$, and 

{\rm (ii)} $\sum_{k=0}^{r-1} \lambda_{jki}=\delta_j$.
\end{lem}

\begin{proof}
(i) By the associativity of multiplication we have the following condition on the structure constants for all $i,j,k,\ell,m\in \{0,1,\ldots,r-1\}$,
\begin{align*}
	\sum_{\ell}^{} \lambda_{ij\ell}\lambda_{\ell km}=\sum_{\ell}^{}\lambda_{i\ell m}\lambda_{jk\ell}
\end{align*}
Now, fix $k$ and let $m=0$. Using  the pseudo-inverse condition on $\mathbf{B}$ we have $\lambda_{jki^*}\delta_i=\lambda_{kij^*}\delta_j=\lambda_{ijk^*}\delta_k$. 

For (ii), we have that for all $i,j \in \{0,1,\dots,r-1\}$, $b_{j^*}b_i = \sum_{k=0}^{r-1} \lambda_{j^*ik}b_k = \sum_{k=0}^{r-1} \lambda_{i^*jk^*}b_k$.  Since $\chi_0(b_j^*)=\chi_0(b_j)$ and the degree map is an algebra homomorphism from $A$ to $\mathbb{C}$, we have $$\delta_j \delta_i = \chi_0(b_{j^*}b_i) =\sum_{k=0}^{r-1} \lambda_{i^*jk^*}\delta_k,$$ 
which is equal by (i) to $\sum_{k=0}^{r-1} \lambda_{jki}\delta_i$.  So, (ii) follows. 
\end{proof}

Note that Lemma \ref{conditions on Sc's} (ii) tells us that every row sum of the left regular matrix of $b_j \in \mathbf{B}$ is equal to the constant $\delta_j$.  

\medskip
Next, we consider restrictions on the parameters of a SITA imposed by its fusions.  If $I \subseteq \{1,\dots,r-1\}$, we let $b_I = \sum_{i \in I} b_i$.  When $\Lambda = \{\{0\}, I_1, \dots, I_{s-1}\}$ is a partition of $\{0,1,\dots,r-1\}$ for which $\mathbf{B}_{\Lambda} = \{b_0, b_{I_1},\dots,b_{I_{s-1}}\}$ is the basis of a table algebra (which will automatically be the standard basis of a SITA in this case), then we say that $\mathbf{B}_{\Lambda}$ is a {\it fusion} of $\mathbf{B}$, and conversely say that $\mathbf{B}$ is a {\it fission} of $\mathbf{B}_{\Lambda}$.  The next lemma shows that every SITA admits a rank 2 fusion.   

\begin{lem}{\label{fission of rank 2}}
Every table algebra $(A,\mathbf{B})$ with standard basis $\mathbf{B}=\{b_0,b_1,\ldots,b_{r-1}\}$ of rank $r\geq3$ has the trivial rank $2$ fusion $\mathbf{B}_{\{\{0\},\{1,\dots,r-1\}\}} = \{b_0, b_1+\dots+b_{r-1} \}$. 
\end{lem}

\begin{proof} 
Let $\mathbf{B}^+ = \sum_{j=0}^{r-1} b_j$.  By \cite{arad1999generalized} we have 
$(\mathbf{B}^+)^2=\chi_0(\mathbf{B}^+)\mathbf{B}^+=n\mathbf{B}^+$.  It follows that 
$((\mathbf{B}-\{b_0\})^+)^2=n\mathbf{B}^+-2\mathbf{B}^++b_0=(n-2)\mathbf{B}^++b_0=(n-2)(\mathbf{B}^+-\{b_0\})+(n-1)b_0.$  This implies $\{b_0, \mathbf{B}^+-b_0 \}$ is a $*$-invariant subset of $\mathbf{B}$ that generates a $2$-dimensional subalgebra of $A$.  The lemma follows. 
\end{proof} 

The conditions imposed by fusion on the parameters of a commutative association scheme were studied by Bannai and Song in \cite{bannai1993character}.  For structure constants the conditions are straightforward, for character table parameters the existence of a fusion imposes certain identities on partial row and column sums of $P$.    Let $\Lambda = \{ \{0\}, J_1, \dots, J_{d-1} \}$ be the partition inducing the fusion $\mathbf{B}_{\Lambda} = \{\tilde{b}_0, \tilde{b}_{J_1}, \dots, \tilde{b}_{J_{d-1}} \}$ of our standard integral table algebra basis $\mathbf{B}$. 
If $E = \{ e_0, e_1, \dots, e_{r-1} \}$ is the basis of primitive idempotents of $A$, then there is a (dual) partition $\Lambda^* = \{ \{0\}, K_1,\dots,K_{d-1} \}$ of $\{0,1,\dots,r-1\}$, unique to the fusion, such that, if $\tilde{e}_0 = e_0$ and $\tilde{e}_{K_i} = \sum_{k \in K_i} e_k$ for $i = 1,\dots,d-1$, then $\tilde{E} = \{\tilde{e}_0, \tilde{e}_{K_1}, \dots, \tilde{e}_{K_{d-1}} \}$ is the basis of primitive idempotents of the algebra $\mathbb{C}\mathbf{B}_{\Lambda}$. 

Let $\tilde{P}$ be the character table of the fusion $\mathbf{B}_{\Lambda}$, so the rows of $\tilde{P}$ are indexed by the irreducible characters $\tilde{\chi}_I$ for $I \in \Lambda^*$, and the columns of $\tilde{P}$ are indexed by the basis elements $b_J$ for $J \in \Lambda$.  Let $\tilde{\delta} = \tilde{\chi}_0$ and $\delta = \chi_0$ be the respective degree maps.   Let $\tilde{\delta}(b_J) = \tilde{k}_J$ for all $J \in \Lambda$, and $\delta(b_j) = k_j$ for all $j \in \{1,\dots,r-1\}$.   Let $\tilde{m}_I$ and $m_i$ denote the multiplicities of $\tilde{\chi}_I$ and $\chi_i$, respectively. Then we have the following identities on partial row and column sums.

\begin{thm}     (Theorem 1.4,   \cite{bannai1993character})\label{thm:fissionfacts}
Let $J \in \Lambda$ and $I \in \Lambda^*$.  
		
{\rm (i)} For all $j \in J$, $\sum_{i \in I} m_i P_{i,j} = \frac{k_j \tilde{m}_I}{\tilde{k}_J} \tilde{P}_{I,J}$. 
		
{\rm (ii)} For all $i \in I$, then $\tilde{P}_{I,J} = \sum_{j \in J} P_{i,j}$.  
\end{thm} 

\begin{proof} (i). We are assuming $\tilde{e}_I = \sum_{i \in I} e_i$.  Using the formula for primitive idempotents in a standard table algebra \cite{arad1999generalized},   
		$$ \tilde{e}_I = \frac{\tilde{m}_I}{n} \sum_J \frac{\tilde{P}_{I,J}}{\tilde{k}_J} \tilde{b}_J^* = \sum_J \sum_{j \in J} \frac{\tilde{m}_I \tilde{P}_{I,J}}{n \tilde{k}_J } b_j^*. $$
		On the other hand, 
		$$ \tilde{e}_I = \sum_{i \in I} e_i = \sum_{i \in I} \frac{m_i}{n} \sum_j \frac{P_{i,j}}{k_j} b_j^* $$
		$$ \qquad = \sum_J \sum_{j \in J} \sum_{i \in I} \frac{m_i P_{i,j}}{n k_j} b_j^*. $$ 
Therefore, for all $j \in J$, $\sum_{i \in I} m_i P_{i,j} = \frac{k_j \tilde{m}_I}{\tilde{k}_J} \tilde{P}_{I,J}$, as required.  
		
		(ii). When $\chi_i(b_0) = 1$, we have $b_je_i = P_{i,j} e_i$ for all $b_j \in \mathbf{B}$.  On the one hand, 
		$$ \tilde{b}_J \tilde{e}_I = \tilde{P}_{I,J} \tilde{e}_I = \sum_{i \in I} \tilde{P}_{I,J} e_i, $$
		and on the other hand, assuming $\chi_i(b_0) = 1$ for all $i \in I$, 
		$$ \tilde{b}_J \tilde{e}_I = \sum_{j \in J} \sum_{i \in I} b_j e_i = \sum_{i \in I} \sum_{j \in J} P_{i,j} e_i. $$
		Therefore, $\tilde{P}_{I,J} = \sum_{j \in J} P_{i,j}$ for all $i \in I$.  
	\end{proof}  

We remark that the fusion condition (i) on partial column sums holds without change for noncommutative table algebras.  Note that standard character considerations tell us $\sum_{j \in J} m_j \chi_j(b_0) = \tilde{m}_J$.  Condition (ii) on partial row sums holds for the rows of $\tilde{P}$ indexed by the $\tilde{\chi}_I$ for which $\chi_i(b_0)=1$ for all $i \in I$. 
		
\subsection{The Splitting Field and its Galois Group}
If $(A,\mathbf{B})$ is a commutative integral table algebra with standard basis $\mathbf{B}=\{b_0, b_1, \cdots, b_{r-1}\}$, the {\it splitting field} of $(A,\mathbf{B})$ is the field $K$ obtained by adjoining all the eigenvalues of the regular matrices of elements of $\mathbf{B}$ to the rational field $\mathbb{Q}$, or equivalently, the smallest field $K$ for which the character table $P$ lies in $M_r(K)$, the algebra of $r \times r$ matrices over the field $K$. As each $b_j$ in $\mathbf{B}$ is a nonnegative integer matrix, $K$ is also the unique minimal Galois extension of $\mathbb{Q}$ that splits the characteristic polynomials of every $b_j \in \mathbf{B}$.  Let $G = Gal(K/\mathbb{Q})$ be the Galois group of this splitting field.  Since the irreducible characters of $A$ are also irreducible representations of $A$ in the commutative case, $G$ will act faithfully on the set of irreducible characters of $A$ via $\chi_i^{\sigma}(b_j)=(\chi_i(b_j))^{\sigma}$, for all $\chi_i \in Irr(A)$, $b_j \in \mathbf{B}$, and $\sigma \in G$.  In this way $G$ permutes the rows of the character table $P$, as well as the corresponding multiplicities.  For SITAwIMs this means $G$ can only permute sets of irreducible characters with the same multiplicity. 

By the Kronecker-Weber theorem, a necessary and sufficient condition for $(A,\mathbf{B})$ to be a standard integral table algebra with noncyclotomic character values is for this Galois group $G$ to be non-abelian.  If $G$ is non-abelian, the fact that the action of $G$ on irreducible characters of $A$ is faithful forces there to be at least one orbit of size $3$ or more.  

\begin{thm}\label{thm:central} \cite{munemasa1991splitting}
Let $(A,\mathbf{B})$ be an integral table algebra (possibly noncommutative).  Let $H$ be the subset of $G=Gal(K/\mathbb{Q})$ consisting of elements $\sigma \in G$ whose action on the character table $P = (\chi(b))_{\chi,b}$ can be realized by a permutation of the basis, that is, for all $b\in \mathbf{B}$ there exists $b^{\sigma}\in \mathbf{B}$ such that for all $\chi \in Irr(A)$, $(P_{\chi,b})^{\sigma} = \chi(b^{\sigma}) = P_{\chi,b^{\sigma}}$.  Then $H$ is a central subgroup of $G$. 
\end{thm}

\begin{proof}
To see that $H$ is a subgroup of $G$, let $\sigma,\tau \in H$, $\chi \in Irr(A)$, and $b\in\mathbf{B}$.  Then
$$ (P_{\chi,b})^{\sigma \tau} = ((P_{\chi,b})^{\sigma})^{\tau} = (P_{\chi,b^{\sigma}})^{\tau} = (P_{\chi,b^{\sigma \tau}}). $$
Therefore, $\sigma\tau \in H$.  Since $G$ is finite, $H$ is a subgroup. 

To see that $H$ is central, let $\tau \in H$, $\sigma \in G$, $\chi \in Irr(A)$, and $b \in \mathbf{B}$.  Then 
$$ (P_{\chi,b})^{\sigma \tau}=(P_{\chi^{\sigma}, b})^{\tau}=(P_{\chi^{\sigma}, b^{\tau}})=(P_{\chi, b^{\tau}})^{\sigma}=((P_{\chi b})^{\tau})^{\sigma}=(P_{\chi, b})^{\tau\sigma}. $$ 
As the action of $G$ on the rows of $P$ is faithful, this implies $\sigma\tau=\tau\sigma$, so $H$ is contained in $Z(G)$. 
\end{proof}

The above theorem always applies to commutative table algebras that are not symmetric.

\begin{cor}\label{asym}
Suppose $(A,\mathbf{B})$ is a commutative table algebra that is not symmetric.  Then the restriction of complex conjugation to $K$ is a nonidentity element of the center of $G$.
\end{cor} 

\begin{proof} 
Commutative table algebras that are not symmetric always have at least one irreducible character that is not real-valued.  If otherwise, the identity $\chi_i(b_j^*)=\overline{\chi_i(b_j)}$, for all $\chi_i \in Irr(A)$ and $b_j \in \mathbf{B}$, would imply the character table $P$ would not be invertible.  For the irreducible characters that are not real-valued, the restriction of complex conjugation to $K$ will be a non-identity element of $G$ that is realized by the permutation of $\mathbf{B}$ corresponding to the involution.  By Theorem \ref{thm:central}, this element lies in the center of $G$. 
\end{proof} 

\subsection{Algebraic sets for SITAwIMs of a given rank and involution type.}
As indicated in the introduction, we will obtain our results by searching for suitable nonnegative integer points in an algebraic set (i.e., the solution set to a system of polynomial equations) that is determined by the parameters of SITAwIMs of a given rank and involution type.  To illustrate how the generating sets for the ideals corresponding to these algebraic sets are produced, we give the type $4A1$ case as an example.  This is the algebraic set corresponding to rank $4$ SITAwIMs whose basis $\mathbf{B}$ contains one asymmetric pair, i.e., $\mathbf{B} = \{b_0,b_1,b_2,b_2^*\}$.  Using the properties of the involution, the row sum property, commutativity of the algebra, and the fact that $b_ib_j=\sum_{k=0}^{3}\lambda_{ijk}b_k$, the general form of the regular matrices for the nontrivial elements of this basis is 
$$ b_1 = \begin{bmatrix} 0 & k _1 & 0 & 0 \\ 1 & k_1-2x_1-1 & x_1 & x_1 \\ 0 & k_1-x_2-x_3 & x_2 & x_3 \\ 0 & k_1-x_2-x_3 & x_3 & x_2 \end{bmatrix}, \quad b_2 = \begin{bmatrix} 0 & 0 & 0 & k_2 \\ 0 & x_1 & k_2-x_1-x_4 & x_4 \\ 1 & x_2 & k_2-x_2-x_5-1 & x_5 \\ 0 & x_3 & k_2-x_3-x_5 & x_5 \end{bmatrix}, \mbox{ and } $$ 
$$ b_2^* = \begin{bmatrix} 0 & 0 & k_2 & 0 \\ 0 & x_1 & x_4 & k_2-x_1-x_4 \\ 0 & x_3 & x_5 & k_2-x_3-x_5 \\ 1 & x_2 & x_5 & k_2-x_2-x_5-1 \end{bmatrix}. $$
Identifying entries in the matrix equations resulting from the identities that define the regular representation gives several linear and quadratic identities in the variables $x_1,\dots,x_5,k_1,k_2$, each of which corresponds to a multivariate polynomial equalling $0$.  For example, identifying entries on both sides of the matrix equation 
$$b_1b_2 = x_1b_1+x_2b_2+x_3b^*_2$$ gives a list of $8$ polynomials: 
$$\begin{array}{l} 
-x_2k_2+x_4k_1, \\
-x_1k_1-x_3k_2-x_4k_1+k_1k_2, \\
-x_1x_3-x_2x_4+x_3^2-x_3x_4+2x_3x_5-x_3k_2+x_4k_1, \\
x_1x_3-x_1k_1+x_2x_4-x_2k_2-x_3^2+x_3x_4-2x_3x_5-x_4k_1+k_1k_2, \\
-x_1x_2+x_2x_3-x_2x_4-x_3x_4+2x_3x_5-x_3k_2+x_4k_1+x_3, \\
x_1x_2-x_1k_1-x_2x_3+x_2x_4-x_2k_2+x_3x_4-2x_3x_5-x_4k_1+k_1k_2-x_3, \\
-x_1^2+x_1x_3-2x_1x_4+2x_1x_5-x_2x_4+x_3x_4-x_3k_2+x_4k_1-x_4+k_2, \mbox{ and } \\
x_1^2-x_1x_3+2x_1x_4-2x_1x_5-x_1k_1+x_2x_4-x_2k_2-x_3x_4-x_4k_1+k_1k_2+x_4-k_2.
\end{array}$$    
We get similar lists of polynomials from the defining identities for $b_1^2$, $b_1b_2^*$, $b_2^2$, $b_2b_2^*$, and $(b_2^*)^2$, and possibly still more from the commuting identities $b_1b_2=b_2b_1$, $b_1b_2^*=b_2^*b_1$, and $b_2b_2^*=b_2^*b_2$.  In the type $4A1$ case, up to sign, this process produces $16$ distinct polynomials.  

When we add the integral multiplicities condition, it leads to extra trace identities that can be added to our list.  For each choice of multiplicities $m_i \in \mathbb{Z}^+$, $i=1,\dots,r-1$, we have an identity satisfied by our character table $P$ resulting from the column orthogonality relation: 
$$ k_j + \sum_{i=1}^{r-1} m_i P_{i,j} = 0. $$ 
In light of assumptions we can make regarding the Galois group, certain rows of $P$ will be Galois conjugate, and the sums of $P_{i,j}$'s corresponding to these rows have to be rational algebraic integers, and thus integers.  The multiplicities corresponding to Galois conjugate rows are the same.  Summing these rows of $P$ gives the rational character table, an integer matrix satisfying certain column and row orthogonality conditions.  The entries in each column of this matrix are bounded in terms of the first entry $k_j$ of the column, so we can search for the possible rational character tables for a given choice of multiplicities.  For each possible rational character table, we can add linear trace identities 
$$ tr(b_j) = k_j + \sum_{i=1}^{r-1} P_{i,j}, \quad j=1,\dots,r-1, $$
to our list of polynomials.  

Let $\mathcal{S}$ be the set of polynomials produced by this process.  Let $\mathcal{I}$ be the ideal generated by $\mathcal{S}$, and let $\mathcal{V}(\mathcal{I})$ be the corresponding algebraic set.  The regular matrices of any SITAwIM of type $4A1$ with the given choice of multiplicities corresponds naturally to a point in $\mathcal{V}(\mathcal{I})$ with $x_1, \dots, x_5 \in \mathbb{N}$ and $k_1,k_2 \in \mathbb{Z}^+$.  We will refer to this as a {\it suitable} integral point in the algebraic set.  Conversely, any suitable integral point in $\mathcal{V}(\mathcal{I})$ corresponds to a SITAwIM of this rank, involution type, and choice of multiplicities. 

For example, if we assume $m_1=m_2=m_3$ in the type $4A1$ case, it adds the trace identities $tr(b_j)=k_j-1$ for $j=1,2,3$, all of which reduce to $x_1=x_2$.  Since this pseudocyclic assumption implies the SITA is homogeneous, we also get $m_1=k_1=k_2$.  Other linear identities, or ones that become linear after cancelling one of our nonzero degrees $k_j$, can also be used to reduce the number of variables we need to consider.  For example, in the type $4A1$ case, one of the elements of $\mathcal{S}$ is $k_2(k_2-1-x_2-2x_5)$, so we can substitute $x_2=k_2-1-2x_5$ and reduce the number of variables by one.  After we reduce by all available linear substitutions in the type $4A1$ case, only one polynomial remains:
$$f(x_5,k_1) = 36x_5^2-24x_5k_1+4k_1^2+32x_5-11k_1+7.$$  
Putting this together with our linear substitutions, we can conclude that any pseudocyclic SITAwIM of type $4A1$ corresponds, via the above regular matrices, to an integer point $(x_1,x_2,x_3,x_4,x_5,k_1,k_2)$ for which $f(x_5,k_2)=0$, $x_5 \ge 0$, $k_1=k_2>0$, $x_1=x_2=x_4=k_1-2x_5-1 \ge 0$, and $x_3=4x_5-k_1+2 \ge 0$.  This is an effective formula to generate pseudocyclic SITAwIMs of type $4A1$.  
  
We refer the readers to  ~\cite{HMprogram} for the GAP implementation that produces
  the defining list of polynomials for rank 4 and 5 SITAwIMs of each
  involution type.

\section{Rank 4 SITAwIMs have cyclotomic eigenvalues}

In this section we show that rank $4$ SITAwIMs have cyclotomic eigenvalues.  In this case there are two involution types to consider: type $4A1$ and type $4S$. 

\begin{prop} 
Rank $4$ SITAwIMs with one asymmetric pair of standard basis elements have cyclotomic eigenvalues.  In fact, their eigenvalues lie in quadratic number fields.   
\end{prop}

\begin{proof} Suppose $(A,\mathbf{B})$ is a SITAwIM of rank $4$ with $\mathbf{B} = \{ b_0, b_1, b_2, b_2^*\}$.  If there were nonidentity elements of $\mathbf{B}$ with noncyclotomic eigenvalues, the Galois group $G$ of the splitting field $K$ would have to be $3$-point transitive; i.e., a transitive subgroup of $Sym(\{\chi_1, \chi_2, \chi_3 \})$.  Since $G$ would have to be non-abelian, it would have to be isomorphic to $S_3$.  But $|Z(G)|>1$ by Corollary \ref{asym}, so this is a contradiction. 

Since there are no $3$-transitive groups with a central element of order $2$, we can conclude that $G$ is cyclic of order $2$, and therefore $K$ is a quadratic extension of $\mathbb{Q}$.
\end{proof}

\begin{thm}{\label{main thm}}
Symmetric rank $4$ SITAwIMs have cyclotomic eigenvalues. 
\end{thm}	
	
\begin{proof}
Suppose $(A,\mathbf{B})$ is a symmetric SITAwIM of rank $4$ that has noncyclotomic eigenvalues.  If $G$ is the Galois group of its splitting field $K$, then as in the rank $4$ one asymmetric pair case, $G$ must act as the full symmetric group on the set $\{ \chi_1, \chi_2, \chi_3 \}$.  In particular this implies these three characters have the same multiplicity $m$.  Therefore, $n = 1 + 3m$, and the character table $P$ of $(A, \mathbf{B})$ has the form 
\begin{table}[H]	
	\begin{longtable}{c|cccc|c} 
		& $b_{0}$  & $b_{1}$ & $b_{2}$ & $b_{3}$ & $\text{multiplicities}$ \\ \hline
		$\chi_{0}$ &$1 $&  $\delta_1$& $\delta_2$   & $\delta_3$ &  $1$ \\ \hline
		$\chi_{1}$ &$1$&  $\alpha_1$& $\beta_1$  & $\gamma_1$ & $m$\\ \hline
		$\chi_{2}$ &$1$ &  $\alpha_2$ & $\beta_2$ & $\gamma_2$ & $m$\\ \hline
		$\chi_{3}$ &$1$ &  $\alpha_3$& $\beta_3$ & $\gamma_3$ & $m$ 
	\label{table:p}	
	\end{longtable} 

\end{table} 	 	
\noindent where $\{\delta_1,\alpha_1,\alpha_2,\alpha_3 \}$,  $\{\delta_2,\beta_1,\beta_2,\beta_3\}$, and $\{\delta_3,\gamma_1,\gamma_2,\gamma_3\}$ are the eigenvalues of $b_1$, $b_2$, and $b_3$, respectively.  If we apply Theorem \ref{thm:fissionfacts} (i) to the column of $P$ labeled by $b_1$, we get  
$$ m(\alpha_1+\alpha_2+\alpha_3)=\frac{\delta_{1}}{n-1}(n-1)(-1), \mbox{ so } \alpha_1+\alpha_2+\alpha_3=\frac{-\delta_{1}}{m}. $$ 
Since $\alpha_1+\alpha_2+\alpha_3$ is an algebraic integer, we must have that $m$ divides $\delta_1$.  Similarly $m$ divides $\delta_2$ and $\delta_3$.  Since $\delta_1+\delta_2+\delta_3 = n-1 = 3m$ we must have $\delta_{1}= \delta_{2} = \delta_{3} = m$.

Assume $\alpha_1$ is a noncyclotomic eigenvalue of $b_1$.   Since $\delta_1$ is an integral eigenvalue of $b_1$, the minimal polynomial $\mu_{\alpha_1}(x)$ of $\alpha_1$ in $\mathbb{Q}[x]$ will be a divisor of $(x-\alpha_1)(x-\alpha_2)(x-\alpha_3)$.   If the degree of $\mu_{\alpha_1}(x)$ is $1$ or $2$, it would follow that $\alpha_1$ is rational or lies in a quadratic extension of $\mathbb{Q}$, which runs contrary to our assumption that it is not cyclotomic.  So $(x-\alpha_1)(x-\alpha_2)(x-\alpha_3)$ is the minimal polynomial of $\alpha_1$ in $\mathbb{Q}[x]$. 
This implies $\mathbb{Q}(\alpha_1,\alpha_2,\alpha_3)$ is the splitting field of $\alpha_1$ over $\mathbb{Q}$.  Since $\alpha_1$ is not cyclotomic, this has to be an extension of $\mathbb{Q}$ with $[\mathbb{Q}(\alpha_1,\alpha_2,\alpha_3):\mathbb{Q}] = 6$.  Since $\mathbb{Q}(\alpha_1,\alpha_2,\alpha_3) \subseteq K$ and $[K:\mathbb{Q}]=|G| = 6$, we must have $K = \mathbb{Q}(\alpha_1,\alpha_2,\alpha_3)$.   

Now consider the left regular matrices of $b_1, b_2, b_3$ in the basis $\mathbf{B}$.  For convenience we write these in this form:  
$$ b_1=\begin{bmatrix}
		0 & m & 0 & 0\\
		1 & u & x_1 & x_4\\
		0 & v & x_2 & x_5\\
		0 & w & x_3 & x_6	
		\end{bmatrix},\,\,\,\,\,\,\,	b_2=\begin{bmatrix}
		0 & 0 & m & 0\\
		0 & x_1 & u' & x_7\\
		1 & x_2 & v' & x_8\\
		0 & x_3 & w' & x_9	
		\end{bmatrix},\,\,\,\,\,\,\,	b_3=\begin{bmatrix}
		0 & 0 & 0 & m\\
		0 & x_4 & x_7 & u''\\
		0 & x_5 & x_8 & v''\\
		1 & x_6 & x_9 & w''	
		\end{bmatrix}, $$
where the $u$, $v$, and $w$ entries are determined by the row sum criterion.  Applying the structure constant identities which define the left regular matrices produces one polynomial identity in the variables $x_1, \dots, x_9, m$ for each entry of the product $b_i b_j$ for $i,j \in \{1,2,3\}$. 

Since $\mathbf{B}$ is pseudocyclic, we have three more trace identities.  On the one hand, we have $tr(b_1) = u + x_2 + x_6 = (m-1-x_1-x_4)+x_2+x_6$, and on the other, $ tr(b_1) = \delta_1 + \alpha_1 + \alpha_2 + \alpha_3 = m + \alpha_1 + \alpha_2 + \alpha_3 = m-1$, so we can restrict our algebraic set by adding the polynomial $x_2+x_6-x_1-x_4$ to our list.  Similar identities coming from $tr(b_2)=m-1$ and $tr(b_3)=m-1$ show we can add the polynomials $x_1+x_9-x_2-x_8$ and $x_4+x_8-x_6-x_9$ to our list.  
 
Next, we reduce our list of polynomials using all available linear substitutions and obtain 
$$\begin{array}{rclclcrclcl} \label{eqa: 4}
x_1 &=& v & = & m-x_2-x_5 & \qquad \qquad & x_6 &=& u'' & = & m-x_4-x_7    \\
x_2 &=& u' & = & m-x_1-x_7 & \qquad & x_8 &=& w' & = & m-x_3-x_9  \\
x_3 &=& x_5 & = & x_7    & \qquad & x_9 & = & v'' & = & m-x_5-x_8.  \\
x_4 &=& w & = & m-x_3-x_6 & \qquad & & & & &  
\end{array}$$ 
This implies the matrix of $b_1$ is 
$$	b_1=\begin{bmatrix}
	0 & m & 0 & 0\\
	1 & u & x_1 & x_4\\
	0 & x_1 & x_2 & x_3\\
	0 & x_4 & x_3 & x_6	
	\end{bmatrix}, $$
so by the row sum criterion $x_1+x_2+x_3=x_4+x_3+x_6=m$, which implies $x_1+x_2=x_4+x_6$.  But the identity we obtained by considering $tr(b_1)$ was $x_1+x_4=x_2+x_6$, so we must conclude that $x_4=x_2$, and hence $x_6=x_1$.  Similarly, we see that the matrix of $b_2$ is 
$$b_2=\begin{bmatrix}
	0 & 0 & m & 0\\
	0 & x_1 & x_2 & x_3\\
	1 & x_2 & v' & x_8\\
	0 & x_3 & x_8 & x_9	
	\end{bmatrix}, $$ 
so $x_3+x_8+x_9=m$, and we must have $x_1+x_2=x_8+x_9$.  Comparing this to the identity $x_1+x_9=x_2+x_8$ obtained by considering $tr(b_2)$, we see that $x_9=x_2$, and it then follows that $x_8=x_1$. 

Therefore, we have 
$$ b_1=\begin{bmatrix}
		0 & m & 0 & 0\\
		1 & x_3-1 & x_1 & x_2\\
		0 & x_1 & x_2 & x_3\\
		0 & x_2 & x_3 & x_1	
		\end{bmatrix}, b_2=\begin{bmatrix}
		0 & 0 & m & 0\\
		0 & x_1 & x_2 & x_3\\
		1 & x_2 & x_3-1 & x_1\\
		0 & x_3 & x_1 & x_2	
		\end{bmatrix}, \mbox{ and } b_3=\begin{bmatrix}
		0 & 0 & 0 & m\\
		0 & x_2 & x_3 & x_1\\
		0 & x_3 & x_1 & x_2\\
		1 & x_1 & x_2 & x_3-1	
		\end{bmatrix}. $$
If we take $Q$ to be the permutation matrix
	\begin{align*}
 Q=\begin{bmatrix}
1 & 0 & 0 & 0\\
0 & 0 & 1 & 0\\
0 & 0 & 0 & 1\\
0 & 1 & 0 & 0	
\end{bmatrix}, 
	\end{align*}
then we have $Q^{-1} b_1 Q= b_2$, $Q^{-1} b_2 Q= b_3$, and $Q^{-1} b_3 Q = b_1$. It follows that the regular matrices of $b_1$, $b_2$, and $b_3$ have the same characteristic polynomial, and that the Galois group $G$ has a nontrivial central element of order $3$ that permutes the corresponding columns in the character table.  But this is contrary to $G$ being isomorphic to $S_3$.  We conclude that for symmetric SITAwIMs of rank $4$, the eigenvalues of basis elements must be cyclotomic.  
\end{proof}	

\begin{cor} 
All association schemes of rank $4$ have cyclotomic eigenvalues. 
\end{cor}
	
\section{Rank $5$ SITAwIMs}

For rank $5$ SITAwIMs we have three involution types to consider: type $5S$, type $5A1$, and type $5A2$.  

\subsection{Type 5A2}
\begin{thm}
	Every rank $5$ SITAwIM $(A,\mathbf{B})$ with $\mathbf{B}=\{b_0,b_1,b_1^*,b_3,b_3^*\}$ has cyclotomic eigenvalues.  
\end{thm}
\begin{proof}
	Let $\mathbf{B} = \{b_0,b_1,b_1^*,b_3,b_3^*\}$ be the standard basis of a SITAwIM of rank $5$, with character table $P$, splitting field $K$, and Galois group $G = \Gal(K/\mathbb{Q})$. 
 As the table algebra is not symmetric, we know by Corollary \ref{asym} that $G$ has a central element of order $2$.  If the character table $P$ has a noncyclotomic entry, then $G$ must also be a $3$- or $4$-point transitive non-Abelian subgroup of $\sym(\{\chi_1,\chi_2,\chi_3,\chi_4\})$, so the only possibility is for $G \simeq D_4$, the dihedral group of order $8$.  This implies the action of $G$ on the last $4$ rows of $P$ is $4$-transitive, and so we must have that the multiplicities $m_1$, $m_2$, $m_3,$ and $m_4$ are all equal to the same positive integer $m$.  So as in the symmetric rank $4$ case, this implies the table algebra is homogeneous: $\delta_1 = \delta_2 = \delta_3 = \delta_4 = m$. 
	
	This implies our regular matrices of $\mathbf{B}$ will have this pattern: 
	
	{\footnotesize
		$$ b_1 = \begin{bmatrix} 0 & 0 & m & 0 & 0 \\ 1 & x_1 & m-1-x_1-x_5-x_9 & x_5 & x_9 \\ 0 & x_2 & m-x_2-x_6-x_{10} & x_6 & x_{10} \\ 0 & x_3 & m-x_3-x_7-x_{11} & x_7 & x_{11} \\ 0 & x_4 & m-x_4-x_8-x_{12} & x_8 & x_{12} \end{bmatrix},
		b_1^* = \begin{bmatrix} 0 & m & 0 & 0 & 0 \\ 0 & m-x_2-x_6-x_{10} & x_2 & x_{10} & x_6 \\ 1 & m-1-x_1-x_5-x_9 & x_1 & x_9 & x_5 \\ 0 & m-x_4-x_8-x_{12} & x_4 & x_{12} & x_8 \\ 0 & m-x_3-x_7-x_{11} & x_3 & x_{11} & x_7  \end{bmatrix}, $$
		$$ b_3 = \begin{bmatrix} 0 & 0 & 0 & 0 & m \\ 0 & x_5 & x_{10} & x_{13} & m-x_5-x_{10}-x_{13} \\ 0 & x_6 & x_9 & x_{14} & m-x_6-x_9-x_{14} \\ 1 & x_7 & x_{12} & x_{15} & m-1-x_7-x_{12}-x_{15} \\ 0 & x_8 & x_{11} & x_{16} & m-x_8-x_{11}-x_{16} \end{bmatrix}, 
		b_3^* = \begin{bmatrix} 0 & 0 & 0 & m & 0 \\ 0 & x_9 & x_6 & m-x_6-x_9-x_{14} & x_{14} \\ 0 & x_{10} & x_5 & m-x_5-x_{10}-x_{13} & x_{13} \\ 0 & x_{11} & x_8 & m-x_8-x_{11}-x_{16} & x_{16} \\ 1 & x_{12} & x_7 & m-1-x_7-x_{12}-x_{15} & x_{15}  \end{bmatrix}. $$
	}   
	
	In addition to the set of polynomial identities in the variables $x_1,\dots,x_{16},m$ we obtain by applying the structure constant identities to these regular matrices, we again have the additional trace identities coming from $tr(b_1) = tr(b_3) = m-1$, which adds the polynomial identities 
	$$ x_1+x_7+x_{12}+1=x_2+x_6+x_{10} \mbox{ and }  x_5+x_9+x_{15}+1=x_8+x_{11}+x_{16} $$
	to our list. The result is a list of $13$ distinct polynomial generators, up to sign, for an ideal of $\mathbb{Q}[x_1,\dots,x_{16},m]$. 	
%	(see Appendix \ref{AP2}). 
 The  available linear substitutions are:  
	\begin{align}\label{Identities}
		\left\{
		\begin{aligned}
			&x_{16} = 2m-6x_1-x_2-2,  \\
			&x_{15}  =  x_1, \\
			&x_{14}  =  x_8 \; = \; 3x_1+x_2+x_3+1-m,  \\
			&x_{13}  =  x_{11} \; = \; 2x_1-x_3+1,  \\
			&x_{12}  =  x_9 \; = \; x_7 \; = \; x_5 \; = \frac{m-1}{2}-x_1,  \\
			&x_{10}  =  x_3,  \\
			&x_6  =  x_4 \; = \; m-x_1-x_2-x_3, 
		\end{aligned}
		\right.
	\end{align}
	so the reduced ideal now lies in $\mathbb{Q}[x_1,x_2,x_3,x_4,m]$.   With the above substitutions, the regular matrices have this pattern: 
	{\footnotesize
		$$ b_1 = \begin{bmatrix} 0 & 0 & m & 0 & 0 \\ 1 & x_1 & x_1 & x_5 & x_5 \\ 0 & x_2 & x_1 & x_4 & x_3 \\ 0 & x_3 & x_5 & x_5 & x_{11} \\ 0 & x_4 & x_5 & x_8 & x_5 \end{bmatrix}, \quad
		b_3 = \begin{bmatrix} 0 & 0 & 0 & 0 & m \\ 0 & x_5 & x_3 & x_{11} & x_5 \\ 0 & x_4 & x_5 & x_8 & x_5 \\ 1 & x_5 & x_5 & x_1 & x_1 \\ 0 & x_8 & x_{11} & x_{16} & x_1  \end{bmatrix}. $$
	}
	When we substitute $x_{16}=x_2+y$ for an extra variable $y$, then reduce using the 
	identities in \eqref{Identities} and calculate the Gr\"obner basis for the resulting ideal with 
	respect to an ordering of variables with $y$ maximal, we find that $y^2$ is one of the elements of the basis.
	
	Therefore, $x_{16}$ must be equal to $x_2$ for all points in our algebraic set.
	Substituting $x_2$ for $x_{16}$ in the first equation of \eqref{Identities} gives $x_2 = m - 3x_1 - 1$, substituting this into the last equation makes $x_4=2x_1 -x_3 +1 = x_{11}$, and substituting $x_2 = m - 3x_1 - 1$ into the third equation gives us $x_8 = x_3$. 
	Hence $b_1$ and $b_3$ have the same characteristic polynomial, so they have the same eigenvalues.  Consequently, $b^*_1$ and $b^*_3$ have the same four eigenvalues as $b_1$.  This implies the Galois group of the splitting field will act transitively on the last four columns of the character table, hence the Galois group will be Abelian.  It follows that any rank $5$ SITAwIM whose standard basis has two distinct asymmetric pairs must have cyclotomic eigenvalues.   	
\end{proof}

\subsection{Type 5A1}

\begin{thm}
Every SITAwIM $(A,\mathbf{B})$ of involution type $5A1$ has cyclotomic eigenvalues. 
\end{thm}

\begin{proof}
Let $\mathbf{B} = \{b_0,b_1,b_2,b_3,b_3^*\}$ be the basis of a SITAwIM of type $5A1$.  By Corollary \ref{asym}, complex conjugation will be realized by a central element of the Galois group $G$ of the splitting field $K$ of $\mathbb{Q}\mathbf{B}$.  If the character table $P$ has an entry which is not cyclotomic, then as in the type $5A2$ case, we must have that $G \simeq D_4$ and acts $4$-transitively on $\{\chi_1, \chi_2, \chi_3, \chi_4 \}$.  It follows that our SITAwIM $(A,\mathbf{B})$ is both pseudocyclic and homogeneous. 

This implies that the pattern for our regular matrices in this case will be: 
	{\footnotesize
		$$ b_1 = \begin{bmatrix} 0 & m & 0 & 0  & 0 \\ 1 & m-1-x_1-2x_5 & x_1 & x_5 & x_5 \\ 0 & m-x_2-2x_6 & x_2 & x_6 & x_6 \\ 0 & m-x_3-x_7-x_{8} & x_3 & x_7 & x_{8} \\ 0 & m-x_4-x_8-x_7 & x_4 & x_8 & x_7 \end{bmatrix}, \quad
		b_2 = \begin{bmatrix} 0 & 0 & m & 0 & 0 \\ 0 & x_1 & m-x_1-2x_9 & x_9 & x_9 \\ 1 & x_2 & m-1-x_2-2x_{10} & x_{10} & x_{10} \\ 0 & x_3 & m-x_3-x_{11}-x_{12} & x_{11} & x_{12} \\ 0 & x_4 & m-x_4-x_{11}-x_{12} & x_{12} & x_{11}  \end{bmatrix}, $$
		
		$$ b_3 = \begin{bmatrix} 0 & 0 & 0 & 0 & m \\ 0 & x_5 & x_{9} & x_{13} & m-x_5-x_{9}-x_{13} \\ 0 & x_6 & x_{10} & x_{14} & m-x_6-x_{10}-x_{14} \\ 1 & x_7 & x_{11} & x_{15} & m-1-x_7-x_{11}-x_{15} \\ 0 & x_8 & x_{12} & x_{16} & m-x_8-x_{12}-x_{16} \end{bmatrix}, \mbox{ and }
	 b_3^* = \begin{bmatrix} 0 & 0 & 0 & m & 0 \\ 0 & x_5 & x_9 & m-x_5-x_9-x_{13} & x_{13} \\ 0 & x_{6} & x_{10} & m-x_6-x_{10}-x_{14} & x_{14} \\ 0 & x_{8} & x_{12} & m-x_8-x_{12}-x_{16} & x_{16} \\ 1 & x_{7} & x_{11} & m-1-x_7-x_{11}-x_{15} & x_{15}  \end{bmatrix}. $$
	}
In addition to the polynomial identities obtained by applying the structure constant identities to these regular matrices, we again have three extra trace identities coming from $tr(b_1) = tr(b_2) = tr(b_3) = m-1$:  
$$ x_1+2x_5=x_2+2x_7, \, x_1+2x_{11}=x_2+2x_{10}, \mbox{ and } x_5+x_{10}+x_{15}+1=x_8+x_{12}+x_{16}.$$
In addition to these, the other available linear substitutions, including those that become linear after we cancel $m > 0$, are: 
$$\begin{array}{rcl}
x_{16} &=& m - x_8 - x_{12} - x_{15}  \\
x_{14} &=& x_{12} \, = \, m - x_6-x_{10}-x_{11}  \\
x_{13} &=& x_8 \, = \, m - x_5 - x_6 - x_7  \\
x_9 & = & x_6 \, = \, x_4 \, = \, x_3 \, = \frac{m}{2} - x_1  \\
x_2 & = & x_1. 
\end{array}$$ 
Since we have the identity $x_3 = \frac{m}{2}-x_1$, integrality of $x_3$ and $x_1$ implies $m=2k$ is even.  
Making as many substitutions as possible, we can leave ourselves with a set of $11$ nonlinear polynomials in $\mathbb{Q}[x_1,x_5,x_{15},m]$.   Using a computer, we calculate the Gröbner basis of the ideal generated by these $11$ polynomials, with $m$ and $x_{15}$ of highest weight.  If we set $y=x_{15}$, the first polynomial in this Gröbner basis is the following element of $\mathbb{Q}[m,y]$:  
$$\begin{array}{rl} W(y,m) = \frac{1}{5184}&(5184y^4-5184y^3m+1944y^2m^2-324ym^3+\frac{81}{4}m^4+7776y^3 -6160y^2m \\
&+1622ym^2-142m^3+4292y^2-2392ym+330m^2+1032y-304m+91). 
\end{array}$$ 
This means $5184 \cdot W(y,m)$ is an integer polynomial that must have a nonnegative solution with $y$ an integer and $m$ an even integer.  But when we substitute $m=2k$, $5184 \cdot W(y,2k)$ has the form $2Q(y,k)+1$ for some polynomial $Q(y,k) \in \mathbb{Z}[y,k]$, and it is impossible for $Q(y,k)=-\frac12$ to have an integral solution.  This implies there are no pseudocyclic SITAwIMs of involution type $5A1$.  In particular this means we can conclude that all rank $5$ SITAwIMs whose standard basis has exactly one asymmetric pair will have cyclotomic eigenvalues. 
\end{proof}

\begin{cor} 
The cyclotomic eigenvalue property holds for every nonsymmetric rank $5$ association scheme. 
\end{cor} 

\subsection{Type 5S} \hfill \\

If $(A,\mathbf{B})$ is a symmetric rank $5$ SITAwIM with noncyclotomic eigenvalues, the action of the Galois group $G =Gal(K/\mathbb{Q})$ of the splitting field $K$ on the irreducible characters of $A$ will either be $3$- or $4$-point transitive.  We begin with the $4$-point transitive case. 

\subsubsection {Type 5S with $4$-point transitive Galois group}
Again in this case we deduce that $(A,\mathbf{B})$ is pseudocyclic and homogeneous from $G$ being $4$-point transitive.  In addition to the polynomial identities obtained by applying the structure constant identities to our regular matrices, we also have four trace identities coming from $tr(b_1)=tr(b_2)=tr(b_3)=tr(b_4)=m-1$.  Altogether our initial list consists of $124$ polynomials in $25$ variables.  By applying all available linear substitutions, we can reduce to a list of $21$ polynomials in $\mathbb{Q}[x_1,x_2,x_3,x_5,x_7,x_{14},x_{15},m]$.  Along the way our first trace identity reduces to 
$$ 2(x_3+x_5+x_{14}-x_{23})=m, $$
so we can conclude that $m$ must be even.  The Gröbner basis of this ideal generated by these $21$ polynomials can be calculated in a few hours on our desktop implementation of GAP \cite{GAP4}, but is too complicated for any easy interpretation.  Instead, reducing to a basis of the rational span of these $21$ polynomials leaves us with just $6$ polynomials.  Using these, we run a search for suitable nonnegative integer solutions, letting $m$ run over increasing even integers and $x_1$, $x_2$, and $x_3$ over the sets of three nonnegative integers that sum to at most $m$.  With these specifications, a Gröbner basis calculation solves for the possible values of the four remaining variables efficiently.  When a suitable nonnegative integer solution is identified, we substitute its values back into our regular matrices and compute the factors of their characteristic polynomials.  Noncyclotomic eigenvalues are detected by applying GAP's {\tt GaloisType} command \cite{GAP4} to irreducible factors of degree $3$ or $4$.  Our searches have found there is only one example with noncyclotomic eigenvalues with $m \le 62$.  We found more examples by carrying out a narrow search with the values of $x_1$, $x_2$, and $x_3$ set to within a 10\% error of $\frac{m}{4}$ for $64 \le m \le 250$.  Up to permutation equivalence, we have found {\it six} symmetric rank $5$ SITAwIMs with $4$-point transitive Galois group that have noncyclotomic eigenvalues.  In all of these cases the Galois group is isomorphic to $S_4$.  (Here we give the factorizations of the characteristic polynomials of their basis elements, from these it is possible to recover the character table $P$ numerically, and from that their other parameters.) 

\medskip 
\centerline{\bf Noncyclotomic SITAwIMs of type $5S$: $4$-point transitive examples} 

{\footnotesize
$$\begin{array}{ll}
n=249: & (x-62)( x^4+x^3-93x^2-57x+12), (x-62)(x^4+x^3-93x^2-306x+261), \\
           & (x-62)(x^4+x^3-93x^2-306x-237), (x-62)(x^4+x^3-93x^2-140x+925) \\
&  \\
n=321: & (x-80)( x^4+x^3-120x^2-341x-242), (x-80)(x^4+x^3-120x^2-20x+2968), \\
           & (x-80)(x^4+x^3-120x^2-301x-400),  (x-80)(x^4+x^3-120x^2+301x+1042) \\
&  \\
n=473: & (x-118)( x^4+x^3-177x^2-266x+279), (x-118)(x^4+x^3-177x^2-266x+3117),  \\
& (x-118)(x^4+x^3-177x^2+680x-667), (x-118)(x^4+x^3-177x^2+207x+4536) \\
&  \\
\end{array}$$
}
{\footnotesize
	$$\begin{array}{ll}
n=633: & (x-158)( x^4+x^3-237x^2-356x+10897), (x-158)(x^4+x^3-237x^2-145x+11108), \\
& (x-158)(x^4+x^3-237x^2+1754x-3451), (x-158)(x^4+x^3-237x^2-778x+5411) \\
&  \\
n=785: & (x-196)( x^4+x^3-294x^2-1619x-1524), (x-196)(x^4+x^3-294x^2-49x+20456), \\
& (x-196)(x^4+x^3-294x^2+1521x+3186), (x-196)(x^4+x^3-294x^2+736x+7896) \\

&  \\
n=993: & (x-248)( x^4+x^3-372x^2+931x-128), (x-248)(x^4+x^3-372x^2+931x+9802), \\
           & (x-248)(x^4+x^3-372x^2+2917x-6086), (x-248)(x^4+x^3-372x^2+1924x+7816). \\
\end{array}$$
}

For all of these examples, the noncyclotomic character table demands a certain algebraic structure of the Wedderburn decomposition of $\mathbb{Q}\mathbf{B}$.  If the character table of $(A,\mathbf{B})$ is $P=(P_{i,j})_{i,j=0}^4 = (\chi_i(b_j))_{i,j=0}^4$, then 
\begin{itemize} 
\item for all $j \in \{1,2,3,4\}$, the four $4$-dimensional primitive extension fields $\mathbb{Q}(P_{1,j})$, $\mathbb{Q}(P_{2,j})$, $\mathbb{Q}(P_{3,j})$, and $\mathbb{Q}(P_{4,j})$ are pairwise distinct and Galois conjugate over $\mathbb{Q}$;
\item for all $i \in \{1,2,3,4\}$, the four primitive extension fields $\mathbb{Q}(P_{i,1})$, $\mathbb{Q}(P_{i,2})$, $\mathbb{Q}(P_{i,3})$, and $\mathbb{Q}(P_{i,4})$ are equal; and
\item for all $i,j \in \{1,2,3,4\}$, $\mathbb{Q}\mathbf{B} \simeq \mathbb{Q} \oplus \mathbb{Q}(P_{i,j})$ as $\mathbb{Q}$-algebras. 
\end{itemize}
Another interesting fact is that the field of Krein parameters will be equal to the splitting field $K$, this is the minimal field of realization for the dual intersection matrices.  

In the last section we explain how to verify that these six SITAwIMs satisfy all the known feasibility conditions for being an association scheme.  The first one is the smallest rank $5$ example with $4$-point transitive Galois group, we present its parameters in detail here.  

\begin{thm}\label{5S-SITAwIM}	
The smallest symmetric rank $5$ SITAwIM with noncyclotomic eigenvalues for which the Galois group of the splitting field is $4$-point transitive has order $249$.  Up to permutation equivalence, its standard basis is given by: 
{\small
	\begin{align*}
	\mathbf{B}=&\left\{
	b_0,\,\,\,\,\,\,\,
	b_1=\begin{bmatrix}
	0 & 62 & 0 & 0 & 0\\
	1 & 15 & 14 & 12 & 20\\
	0 & 14 & 16 & 17 & 15\\
	0 & 12 & 17 & 18 & 15\\
	0 & 20 & 15 & 15 & 12	
	\end{bmatrix},\,\,\,\,\,\,\,	b_2=\begin{bmatrix}
	0 & 0 & 62 & 0 & 0\\
	0 & 14 & 16 & 17 & 15\\
	1 & 16 & 18 & 16 & 11\\
	0 & 17 & 16 & 11 & 18\\
	0 & 15 & 11 & 18 & 18\\	
	\end{bmatrix}\right. \\ &\left. 	b_3=\begin{bmatrix}
	0 & 0 & 0 & 62 & 0\\
	0 & 12 & 17 & 18 & 15\\
	0 & 17 & 16 & 11 & 18\\
	1 & 18 & 11 & 18 & 14\\
	0 & 15 & 18 & 14 & 15\\	
	\end{bmatrix},\,\,\,\,\,\,\, b_4=\begin{bmatrix}
	0 & 0 & 0 & 0 & 62\\
	0 & 20 & 15 & 15 & 12\\
	0 & 15 & 11 & 18 & 18\\
	0 & 15 & 18 & 14 & 15\\
	1 & 12 & 18 & 15 & 16	
	\end{bmatrix}
	\right\}, 
	\end{align*}
}
The character table of $(A,\mathbf{B})$ is shown below.  The roots of the degree $4$ polynomials above have been approximated to six significant digits using {\verb|Wolfram|}$\mid${\verb|Alpha|} \cite{WolframAlphaLLC}.  
$$ P = \begin{bmatrix} 1 & 62 & 62 & 62 & 62 \\
1 & 9.45706 & -4.83450 & -8.21429 & 2.59173 \\
1 & 0.165779 & -7.32957 & 10.6401 & -4.47634 \\
1 & -0.777430 & 10.45989& -2.18457 & -8.49789\\
1 & -9.84541 & 0.704180 & -1.24127& 9.38250 \end{bmatrix}.$$ 

Since this algebra is self-dual, the second eigenmatrix is obtained by setting $Q_{i,j} = P_{j,i}$ for $i=1,2,3,4$ and leaving the first row and column alone. 

The nontrivial dual intersection matrices are as follows, with irrational entries approximated to six significant digits:  
{\footnotesize
$$ L_1^* = \begin{bmatrix} 0 & 62 & 0 & 0 & 0 \\ 1 & 16.2247 & 17.5718 & 15.3191 & 11.8843 \\ 
							0 & 17.5718 & 10.8695 & 18.0841 & 15.4745 \\ 0 & 15.3191 & 18.0841 & 13.9307 & 14.6661 \\ 
                         0 & 11.8843 & 15.4745 & 14.6661 & 19.9751 \end{bmatrix}, \,\,
L_2^* = \begin{bmatrix} 0 & 0 & 62 & 0 & 0 \\ 0 & 17.5718 & 10.8695 & 18.0841 & 15.4745 \\
              1 & 10.8695 & 18.3339 & 16.0173 & 15.7793 \\ 0 & 18.0841 & 16.0173 & 11.1233 & 16.7753 \\
				0 & 15.4745 & 15.7793 & 16.7753 & 13.9710 \end{bmatrix}, $$
$$ L_3^*= \begin{bmatrix} 0 & 0 & 0 & 62 & 0 \\ 0 & 15.3191 & 18.0841 & 13.9307 & 14.6661 \\
                 0 & 18.0841 & 16.0173 & 11.1233 & 16.7753 \\ 1 & 13.9307 & 11.1233 & 17.5255 & 18.4206 \\
				   0 & 14.6661 & 16.7753 & 18.4206 & 12.1381 \end{bmatrix}, \mbox{ and }
L_4^* = \begin{bmatrix} 0 & 0 & 0 & 0 & 62 \\ 0 & 11.8843 & 15.4745 & 14.6661 & 19.9751 \\ 
              0 & 15.4745 & 15.7793 & 16.7753 & 13.9710 \\ 0 & 14.6661 & 16.7753 & 18.4206 & 12.1381 \\
 				1 & 19.9751 & 13.9710 & 12.1381 & 14.9159 \end{bmatrix}.$$
}
\end{thm}

\begin{rmk} One might ask if there are metric association schemes of rank $5$ with noncyclotomic splitting fields that  have $4$-point transitive Galois groups.  With our method, this can be resolved by setting $x_4, x_5, x_9, x_{10}, x_{17}=0$, calculating the Gröbner basis, and using known intersection array restrictions to bound tridiagonal entries of $b_1$.  This approach allows one to make the same conclusion as Blau and Xu  obtain for pseudocyclic metric association schemes in general, that the intersection array has to be $[2,1,1,1;1,1,1,1]$ \cite[Theorem 5.4]{xu2011pseudocyclic}.  But the splitting field of this association scheme has a $3$-point transitive abelian Galois group, so the answer is no.  \end{rmk}

\subsubsection {Type 5S with $3$-point transitive Galois group}
The other possibility for a symmetric SITAwIM of rank $5$ with noncyclotomic eigenvalues is the case where the Galois group of the splitting field is non-abelian and acts $3$-point transitively, so must be isomorphic to $S_3$.  Let $(A,\mathbf{B})$ be such a SITAwIM, with splitting field $K$ and Galois group $G$, and suppose the orbits of $G$ on the irreducible characters of $A$ are $\{\chi_0\}$, $\{\chi_1\}$, and $\{ \chi_2, \chi_3, \chi_4 \}$.  In this situation the table algebra is not necessarily pseudocyclic, nor does it have to be homogeneous, so we do not have as many linear substitutions available to reduce our algebraic set initially.  Instead, to find the SITAwIMs of a given order, we can first make a list of possible rationalized character tables for SITAwIMs of that order.  The rationalized character table is an {\it integer} matrix with columns indexed by $\mathbf{B}$ and rows are indexed by the sums of irreducible characters of $A$ up to Galois conjugacy over $\mathbb{Q}$.  In our $3$-point transitive case, it takes this form:   
\begin{table}[H]	

\begin{longtable}{c|ccccc|c} 
	& $b_{0}$ & $b_{1}$ & $b_{2}$ & $b_{3}$ & $b_4$ & $\text{multiplicities}$ \\ \hline
	$\chi_{0}$ &$1 $&  $\delta_1$ & $\delta_2$  & $\delta_3$ & $\delta_4$ &  $1$ \\ 
	$\chi_{1}$ &$1$&  $a_1$ & $a_2$  & $a_3$ & $a_4$ & $m_1$\\ 
$\chi_2+\chi_3+\chi_4$ & $3$ & $t_1$ & $t_2$ & $t_3$ & $t_4$ & $3m_2$
	\label{table:p}	
\end{longtable} 
\end{table} 	 
The rows and columns of the rationalized character table satisfy orthogonality relations induced by those of the usual character table.  In our case the orthogonality relations give the following identities: 

\begin{itemize} 
\item $\delta_1+\delta_2+\delta_3+\delta_4= m_1 + 3m_2 = n-1$; 
\item $a_1 + a_2 + a_3 + a_4 = -1$; 
\item $t_1+t_2+t_3+t_4=-3$; 
\item $1 + \frac{a_1^2}{\delta_1} + \frac{a_2^2}{\delta_2} + \frac{a_3^2}{\delta_3}+\frac{a_4^2}{\delta_4} = \frac{n}{m_1};$
\item $3 + \frac{a_1t_1}{\delta_1} + \frac{a_2t_2}{\delta_2} + \frac{a_3t_3}{\delta_3}+\frac{a_4t_4}{\delta_4} = 0;$
\item $\delta_1+m_1a_1+m_2t_1=0$; 
\item $\delta_2+m_1a_2+m_2t_2=0$;
\item $\delta_3+m_1a_3+m_2t_3=0$; and 
\item $\delta_4+m_1a_4+m_2t_4=0$.
\end{itemize}

These identities are subject to the restrictions $1 \le m_1, m_2, \delta_1, \delta_2, \delta_3, \delta_4$, and $-\delta_i \le a_i \le \delta_i$ for $i=1,2,3,4$, and $-3\delta_i \le t_i \le 3\delta_i$ for $i=1,2,3,4$.  So a straightforward search will produce all the rationalized character tables possible whose associated SITAwIM would have degree $n$.     

Given a rationalized character table, we get four linear trace identities $tr(b_i)=\delta_i+a_i+t_i$, $i=1,2,3,4$ that can be added to our list of polynomial generators.  This helps us to reduce our search space enough to allow the search and Gröbner basis calculations techniques to uncover suitable nonnegative solutions to the system and produce regular matrices for a SITAwIM with this rationalized character table.  This approach has two computational barriers, which have limited our ability to guarantee a complete account only for orders up to $100$.  First, since we must consider every possibility for $m_1$ and $m_2$ with $1+m_1+3m_2=n$, the number of possible rational character tables of a given order can be very large and time-consuming to generate, and for almost all of these we find no SITAwIM.  Secondly, the values of the $x_i$'s are not as limited as they are in the homogeneous case, so when the minimum $\delta_i$ is large, the search space for all the values of $x_1$, $x_2$, and $x_3$ we need to check grows in size exponentially. 

Our complete search for orders up to $100$ found six examples.   
Their multiplicities and factorizations of the characteristic polynomials of their basis elements are as follows: 

{\center {\bf Noncyclotomic SITAwIMs of type $5S$: $3$-point transitive examples}}

{\footnotesize
$$\begin{array}{ll} 
n=35: & m_1=4, m_2=10, \mu_{b_i}(x) = (x-4)(x+1)(x^3-6x+2), (x-6)^2(x+1)^3, \\
         & (x-12)(x+3)(x^3-12x-2), (x-12)(x+3)(x^3-12x+12); \\
& \\
n=45: & m_1=8, m_2=12, \mu_{b_i}(x) = (x-4)^2(x+1)^3, (x-8)(x+1)(x^3-12x+14), \\
          &(x-8)(x+1)(x^3-12x+4), (x-24)(x+3)(x^3-18x+18); \\
& \\
n=76: & m_1=18, m_2=19, \mu_{b_i}(x) = (x-3)^2(x+1)^3, (x-18)(x+1)(x^3-27x-18), \\
          & (x-18)(x+1)(x^3-27x-42), (x-36)(x+2)(x^3-36x-48); \\
& \\
n=88^a: & m_1=66, m_2=7, \mu_{b_i}(x) = (x-3)^4(x+1), (x-14)(x)(x^3+2x^2-72x-16), \\
         & (x-35)(x)(x^3+5x^2-120x-360), (x-35)(x)(x^3+5x^2-120x+80); \\
& \\
n=88^b: & m_1=66, m_2=7, \mu_{b_i}(x) = (x-3)^4(x+1), (x-21)(x)(x^3+3x^2-96x-384), \\
          & (x-21)(x)(x^3+3x^2-96x-472), (x-42)(x)(x^3+6x^2-120x-784); \mbox{ and } \\
& \\ 
n=93: & m_1=2, m_2=30, \mu_{b_i}(x) = (x-12)(x+6)(x^3-15x+2), (x-20)(x+10)(x^3-21x-16), \\ 
         & (x-30)(x+15)(x^3-24x+8), (x-30)^2(x+1)^3. \\
\end{array}$$
}
Narrow searches of orders $101$ to $250$, the first with $\delta_1 \le 4$ and at least two of $\delta_2$, $\delta_3$, and $\delta_4$ equal, and the second with $\delta_1 \le 12$, $a_1=k_1$, and at least two of $\delta_2$, $\delta_3$, and $\delta_4$ equal produced a few more examples: 

{\footnotesize
$$\begin{array}{ll} 
n=116: & m_1=58, m_2=19, \mu_{b_i}(x) = (x-1)^4(x+1), (x-19)(x)(x^3+x^2-48x+72) \\
           & (x-19)(x)(x^3+x^2-48x-44), (x-76)(x)(x^3+4x^2-72x-32); \\
& \\
n=129: & m_1=86, m_2=14, \mu_{b_i}(x) = (x-2)^4(x+1), (x-28)(x)(x^3+2x^2-99x+150) \\
           & (x-28)(x)(x^3+2x^2-99x-108), (x-70)(x)(x^3+5x^2-135x-75); \\
& \\
n=165: & m_1=32, m_2=44, \mu_{b_i}(x) = (x-4)^2(x+1)^3, (x-32)(x+1)(x^3-48x-32), \\ 
           & (x-32)(x+1)(x^3-48x-112), (x-96)(x+3)(x^3-72x-144); \\
& \\ 
n=189: & m_1=20, m_2=56, \mu_{b_i}(x) = (x-8)^2(x+1)^3, (x-20)(x+1)(x^3-30x-20), \\
           & (x-80)(x+4)(x^3-75x+70), (x-80)(x+4)(x^3-75x-200); \\
&\\

n=190: & m_1=18, m_2=57, \mu_{b_i}(x) = (x-9)^2(x+1)^3, (x-36)(x+2)(x^3-48x+32), \\
        	& (x-36)(x+2)(x^3-48x+112), (x-108)(x+6)(x^3-72x+144); \\
& \\
n=217: & m_1=30, m_2=62, \mu_{b_i}(x) = (x-6)^2(x+1)^3, (x-60)(x+2)(x^3-75x-100), \\
        	& (x-60)(x+2)(x^3-75x-170), (x-90)(x+3)(x^3-90x-180); \\
& \\
n=231^a: & m_1=32, m_2=66, \mu_{b_i}(x) = (x-6)^2(x+1)^3, (x-32)(x+1)(x^3-48x-96), \\
           	& (x-96)(x+3)(x^3-96x-352), (x-96)(x+3)(x^3-96x-128); \\
& \\
n=231^b: &  m_1=32, m_2=66, \mu_{b_i}(x) = (x-6)^2(x+1)^3, (x-32)(x+1)(x^3-48x+16), \\
             	& (x-96)(x+3)(x^3-96x-128), (x-96)(x+3)(x^3-96x+208); \\
\end{array}
$$
}

\begin{exa} The smallest noncyclotomic symmetric rank $5$ SITAwIM with order $n=35$ has regular matrices $b_0$, 
	
{\footnotesize
$$ b_1 = \begin{bmatrix} 0 & 4 & 0 & 0 & 0 \\ 1 & 0 & 0 & 0 & 3 \\ 0 & 0 & 0 & 2 & 2 \\ 0 & 0 & 1 & 2 & 1 \\ 0 & 1 & 1 & 1 & 1 \end{bmatrix}, b_2 = \begin{bmatrix} 0 & 0 & 6 & 0 & 0 \\ 0 & 0 & 0 & 3 & 3 \\ 1 & 0 & 5 & 0 & 0 \\ 0 & 1 & 0 & 2 & 3 \\ 0 & 1 & 0 & 3 & 2 \end{bmatrix}, b_3 = \begin{bmatrix} 0 & 0 & 0 & 12 & 0 \\ 0 & 0 & 3 & 6 & 3 \\ 0 & 2 & 0 & 4 & 6 \\ 1 & 2 & 2 & 4 & 3 \\ 0 & 1 & 3 & 3 & 5 \end{bmatrix}, \mbox{ and } b_4 = \begin{bmatrix} 0 & 0 & 0 & 0 & 12 \\  0 & 3 & 3 & 3 & 3 \\ 0 & 2 & 0 & 6 & 4 \\ 0 & 1 & 3 & 3 & 5 \\ 1 & 1 & 2 & 5 & 3  \end{bmatrix}. $$}

Its first and second eigenmatrices  (with irrationals approximated to six significant digits) are as follows: 

{\footnotesize
$$ P = \begin{bmatrix} 1 & 4 & 6 & 12 & 12 \\ 1 & -1 & 6 & -3 & -3 \\ 
1 & -2.60168 & -1 & -0.167055 & 2.768734 \\ 1 & 0.339877 & -1 & 3.54461 & -3.88448 \\ 
1 & 2.26180 & -1 & -3.37755 & 1.11575  \end{bmatrix}, \mbox{and } 
Q= \begin{bmatrix} 1 & 4 & 10 & 10 & 10 \\ 1 & -1 & -6.50420 & 0.849692 & 5.65451 \\ 
1 & 4 & -5/3 & -5/3 & -5/3 \\ 1 & -1 & -0.139212 & 2.95384 & -2.81463 \\
1 & -1 & 2.30728 & -3.23707 & 0.929791 \end{bmatrix}. $$}

Its dual intersection matrices, again with irrational entries approximated to six significant digits, are: $L_0^* = b_0, $

{\footnotesize
$$ L_1^* = \begin{bmatrix} 0 & 4 & 0 & 0 & 0 \\ 1 & 3 & 0 & 0 & 0 \\ 
0 & 0 & 2/3 & 5/3 & 5/3 \\ 0 & 0 & 5/3 & 2/3 & 5/3 \\ 
0 & 0 & 5/3 & 5/3 & 2/3 \end{bmatrix}, L_2^* = \begin{bmatrix} 0 & 0 & 10 & 0 & 0 \\ 0 & 0 & 5/3 & 25/6 & 25/6 \\ 
1 & 2/3 & 0.0541562 & 2.59972 & 5.67949 \\ 0 & 5/3 & 2.59972 & 3.51139 & 20/9 \\ 
0 & 5/3 & 5.67946 & 20/9 & 0.431651 \end{bmatrix}, $$
$$ L_3^* = \begin{bmatrix} 0 & 0 & 0 & 10 & 0 \\ 0 & 0 & 25/6 & 5/3 & 25/6 \\
                 0 & 5/3 & 2.59972 & 3.51139 & 20/9 \\ 1 & 2/3 & 3.51139 & 2.50545 & 2.31644 \\
					0 & 5/3 & 20/9 & 2.31169 & 3.79463 \end{bmatrix}, \mbox{ and } 
L_4^* = \begin{bmatrix} 0 & 0 & 0 & 0 & 10 \\ 0 & 0 & 25/6 & 25/6 & 5/3 \\ 
              0 & 5/3 & 5.67946 & 20/9 & 0.431651 \\ 0 & 5/3 & 20/9 & 2.31649 & 3.79463 \\
 				1 & 2/3 & 0.431651 & 3.79463 & 4.10706 \end{bmatrix}. $$}\\
\end{exa}

\section{Checking feasibility} 

In this section we review the feasibility checks we have applied to the parameters of the noncyclotomic symmetric rank $5$ SITAwIMs identified in the previous section.   The parameters include the regular (a.k.a.~intersection) matrices $b_i$ $(i \in \{0,1,\dots,r-1\})$, the character table (first eigenmatrix) $P$, the dual character table (second eigenmatrix) $Q$, and the dual intersection matrices (Krein parameters) $L_i^* = (\kappa_{ijk})_{k,j=0}^{r-1}$ $(i \in \{0,1,\dots,r-1\})$.  For commutative association schemes, we consider these to be equivalent since knowledge of any one of these determines the other.  

We have tested our examples on the following feasibility conditions, which apply to general symmetric association schemes: 

\begin{itemize} 
\item the handshaking lemma: for $i,j \in \{1,\dots,r-1\}$, if $i \ne j$, then $(b_i)_{i,j}k_j$ must be even (see \cite[Lemma 7]{herman-muzychuk-xu2018}); 
\item realizability of all closed subsets and quotients; 
\item the triangle count condition: for $j=1,\dots,r-1$, $\displaystyle{\frac{1}{6}\sum_{i=0}^{r-1}} m_i P_{i,j}^3 = t \in \mathbb{N}$; 
\item the absolute bound condition: for $i \in \{0,\dots,r-1\}$, $\displaystyle{\sum_{k; q_{ijk}\ne 0}} m_k \le \begin{cases} m_i m_j & i \ne j \\ {{m_i+1}\choose{2}} & i = j \end{cases}$;
\item nonnegativity of Krein parameters: $(L_i^*)_{j,k} \ge 0$ for $i,j,k \in \{0,\dots,r-1\}$; and 
\item Martin and Kodalen's Gegenbauer polynomial criterion (see \cite[Theorem 3.7 and Corollary 3.8]{kodalen2019}). 
\end{itemize}
 We are aware of one more feasibility condition for symmetric association schemes, the {\it forbidden quadruple} condition described in \cite[Corollary 4.2]{gavrilyuk-vidali-williford2021}. Our $4$-point transitive examples do not have any nontrivial Krein parameters equal to zero, so they satisfy this condition vacuously.  This is not the case for our $3$-point transitive examples, to date these have not been tested for this condition.    

We have ordered these feasibility conditions according to the ease we are able to check them.  Since our algorithms require the multiplicities as part of the input and produce the intersection matrices, we have to compute $P$, then $Q$, then the dual intersection matrices in order from there.  As our objective is only to report the examples that pass all conditions, once an example fails one of our conditions below it is removed and its status for subsequent conditions is not reported.  

\medskip
We will indicate our examples from the previous section by Galois group action and order: $3pt35$, $3pt45$, etc. Recall that $3pt35$ means the $3$-point transitive example of order 35.

\subsection{Handshaking lemma condition:} Only five of our examples have nontrivial basis elements of odd degree, of these five, three of them fail the handshaking lemma condition: $3pt88^a$, $3pt88^b$, and $3pt116$. $3pt76$ and $3pt190$ pass despite having a nontrivial basis element of odd degree.  

\subsection{Realizability of closed subsets and quotients:} All of our $4$-point transitive examples are primitive, so there are no closed subsets or quotients to consider.  On the other hand, all of the remaining $3$-point transitive examples have a unique nontrivial closed subset of rank $2$.  For all but one of these, the quotient also has rank $2$.  The exception is $3pt129$, for which the quotient has rank $4$.  Since this quotient table algebra has an element of non-integral degree, it is not realizable as an association scheme. 

\subsection{Triangle count condition.} All of our examples pass.

\subsection{Absolute bound condition.}  All of our $4$-point transitive examples pass. We can see from the multiplicities that $3pt45$, $3pt76$, and $3pt165$ will pass.   $3pt35$ could potentially fail for $i=j=1$ but passes because $\kappa_{1,1,k}=0$ for $k=2,3,4$. $3pt93$ and $3pt129$ also pass because enough nontrivial Krein parameters are $0$.  

\subsection{Nonnegative Krein parameter condition.} For all of our $3$- and $4$-point transitive examples, we have calculated the dual intersection matrices and found them to be nonnegative.   

\subsection{Gegenbauer polynomial condition.} We check that $G^{m_i}_{\ell}(\frac{1}{m_i}L_i^*)$ is a nonnegative matrix for all $\ell \ge 1$ and $i=1,\dots,4$ using the approach of \cite[\S 3.3]{kodalen2019}.  

We illustrate the process of checking this condition with $3pt35$.  In the case $m_1=4$, it is not possible to find an $\ell^*$ satisfying the conditions of \cite[Corollary 3.16]{kodalen2019}. However, $L_1^*$ is a block matrix, and the upper left $2 \times 2$ block $\begin{bmatrix} 0 & 4 \\ 1 & 3 \end{bmatrix}$ is the dual intersection matrix corresponding to the association scheme generated by the complete graph of order $5$, in which it also occurs with nontrivial multiplicity $4$.  It follows that the first column of $G^{m_i}_{\ell}(\frac{1}{m_i}L_i^*)$ will always be nonnegative for all $\ell \ge 1$, so the result follows by \cite[Corollary 3.8]{kodalen2019} and the remark following it.    

In the cases $m_2 =m_3=m_4=10$, we find that the minimum $\ell^*$ required for \cite[Corollary 3.16]{kodalen2019} is $\ell^*=6$, and we can check that $G^{10}_{\ell}(\frac{1}{10}L_i^*)$ has nonnegative entries for all $\ell \in \{1,\dots,7\}$ and all $i = 2,3,4$.  So, $3pt35$ passes all the feasibility conditions, with the possible exception of the forbidden quadruple condition.   

In all of the remaining $3$-transitive examples, $\mathbf{B}^*$ contains a rank $2$ closed subset of order $m_i+1$ for one $i$.  So, a similar argument as in the $3pt35$ case applies for this $m_i$.  For the other $m_i$ a suitable $\ell^*$ can be found.  After evaluating the appropriate Gegenbauer polynomials at $\frac{1}{m_i}L_i^*$, we found the result to be a nonnegative matrix. 

All of our $4$-point transitive examples pass the Gegenbauer polynomial test.  In each case we have found a value of $\ell^*$ and shown all of the required evaluations result in nonnegative matrices. 

In summary, we have verified that the six $4$-point transitive examples pass all of the feasibility conditions, and ten of the $3$-point transitive examples pass them: $3pt35$, $3pt45$, $3pt76$, $3pt93$, $3pt165$, $3pt189$, $3pt190$, $3pt217$, $3pt231^a$, and $3pt231^b$. Note that by the partial classification of association schemes of order $35$ and rank $5$ in \cite{AH-IM}, we know $3pt35$ cannot be realized. 

\noindent{\sc Acknowledgment:} We would like to thank  anonymous referees for carefully reading the manuscript and for their insightful comments.

\end{document}